\documentclass{amsart}

\usepackage{amsmath, amssymb, amsthm, mathrsfs, graphicx}
\usepackage{hyperref, enumitem, xspace, ifthen,comment}
\usepackage[all]{xy}
\usepackage{tikz}
\usetikzlibrary{arrows,shapes,trees}

\newtheorem*{thm-plain}{Theorem}
\newtheorem{thm}{Theorem}[section]
\newtheorem{lem}[thm]{Lemma}
\newtheorem{prp}[thm]{Proposition}
\newtheorem{cor}[thm]{Corollary}

\newtheorem{conj}[thm]{Conjecture}
\newtheorem{ques}[thm]{Question}
\numberwithin{equation}{thm}

\theoremstyle{definition}
\newtheorem{dfn}[thm]{Definition}

\newtheorem*{dfn-plain}{Definition}

\theoremstyle{remark}
\newtheorem{clm}[thm]{Claim}

\newtheorem{awlog}[thm]{Additional Assumption}

\newtheorem{rem}[thm]{Remark}

\newtheorem*{rem-plain}{Remark}

\DeclareMathOperator{\img}{im}

\DeclareMathOperator{\Spec}{Spec}

\DeclareMathOperator{\codim}{codim}
\DeclareMathOperator{\tor}{tor}
\DeclareMathOperator{\cotor}{cotor}

\DeclareMathOperator{\res}{res}

\DeclareMathOperator{\coker}{coker}

\def\rd#1.{\lfloor{#1}\rfloor}
\def\rp#1.{\lceil{#1}\rceil}

\newcommand{\lto}{\longrightarrow}

\newcommand{\N}{\mathbb N}
\newcommand{\Z}{\mathbb Z}
\newcommand{\Q}{\ensuremath{\mathbb Q}}
\newcommand{\R}{\mathbb R}
\newcommand{\C}{\mathbb C}

\renewcommand{\P}{\mathbb P}
\renewcommand{\O}{\mathscr O}

\newcommand{\x}{\times}

\renewcommand{\phi}{\varphi}
\renewcommand{\theta}{\vartheta}
\newcommand{\id}{\mathrm{id}}

\newcommand{\inj}{\hookrightarrow}
\newcommand{\surj}{\twoheadrightarrow}

\newcommand{\isom}{\cong}
\newcommand{\mf}{\mathfrak}

\newcommand{\mr}{\mathrm}

\newcommand{\tensor}{\otimes}
\newcommand{\Hn}{\mathit\Gamma}

\DeclareMathOperator{\sHom}{\mathscr H\!om}

\newcommand{\sg}{\mathrm{sg}}
\newcommand{\sm}{\mathrm{sm}}

\newcommand{\an}{\mathit{an}}

\newcommand{\Omt}{\check\Omega}
\DeclareMathOperator{\pd}{pd}
\DeclareMathOperator{\depth}{depth}

\newcommand{\sE}{\mathscr{E}}
\newcommand{\sF}{\mathscr{F}}
\newcommand{\sG}{\mathscr{G}}

\newcommand{\sO}{\mathscr{O}}

\newcommand{\sT}{\mathscr{T}}

\newdir{ ir}{{}*!/-5pt/@^{(}} 
\newdir{ il}{{}*!/-5pt/@_{(}} 

\newcommand{\eps}{\varepsilon}

\newcommand{\dif}{\mathrm d}
\renewcommand{\d}{\mathrm d}
\newcommand{\del}{\partial}

\newcommand{\wt}{\widetilde}
\newcommand{\wh}{\widehat}

\newcommand{\GL}{\mathrm{GL}}

\newenvironment{sequation}{%
\setcounter{equation}{\value{thm}}%
\numberwithin{equation}{section}%
\begin{equation}%
}{%
\end{equation}%
\numberwithin{equation}{thm}%
\addtocounter{thm}{1}%
}

\numberwithin{equation}{thm}

\DeclareRobustCommand{\SkipTocEntry}[5]{}

\setitemize[1]{leftmargin=*,parsep=0em,itemsep=0.125em,topsep=0.125em}
\setenumerate[1]{leftmargin=*,parsep=0em,itemsep=0.125em,topsep=0.125em}

\newcommand{\iref}[3]{\the\value{#1}.\the\value{#2}(\the\value{#3})}

\newcommand\factor[2]{\left. \raise 2pt\hbox{$#1$} \right/\hskip -2pt \raise -2pt\hbox{$#2$}}

\definecolor{forrest}{RGB}{81,133,49}
\definecolor{mydarkblue}{RGB}{10,92,153}

\begin{document}

\title{The generalized Lipman--Zariski problem}
\author{Patrick Graf} %
\address{Lehrstuhl f\"ur Mathematik I, Universit\"at Bay\-reuth,
  95440 Bayreuth, Germany} %
\email{\href{mailto:patrick.graf@uni-bayreuth.de}{patrick.graf@uni-bayreuth.de}}
\date{October 16, 2014}
\thanks{The author was supported in full by the
  DFG-Forschergruppe 790 ``Classification of Algebraic Surfaces and Compact Complex
  Manifolds''.}  %

\begin{abstract}
We propose and study a generalized version of the Lipman--Zariski conjecture:
let $(x \in X)$ be an $n$-dimensional singularity such that for some integer $1 \le p \le n - 1$, the sheaf $\Omega_X^{[p]}$ of reflexive differential $p$-forms is free. Does this imply that $(x \in X)$ is smooth?
We give an example showing that the answer is \emph{no} even for $p = 2$ and $X$ a terminal threefold. However, we prove that if $p = n - 1$, then there are only finitely many log canonical counterexamples in each dimension, and all of these are isolated and terminal.
As an application, we show that if $X$ is a projective klt variety of dimension $n$ such that the sheaf of $(n-1)$-forms on its smooth locus is flat, then $X$ is a quotient of an Abelian variety.

On the other hand, if $(x \in X)$ is a hypersurface singularity with singular locus of codimension at least three, we give an affirmative answer to the above question for any $1 \le p \le n - 1$. The proof of this fact relies on a description of the torsion and cotorsion of the sheaves $\Omega_X^p$ of K\"ahler differentials on a hypersurface in terms of a Koszul complex. As a corollary, we obtain that for a normal hypersurface singularity, the torsion in degree $p$ is isomorphic to the cotorsion in degree $p - 1$ via the residue map.
\end{abstract}

\maketitle

\tableofcontents

\section{Introduction}

\subsection{Motivation}

The Lipman-Zariski conjecture \cite{Lip65} asserts the following.

\begin{conj} \label{conj:LZ}
  Let $X$ be a complex variety such that the tangent sheaf $\sT_X := \sHom_{\O_X}(\Omega_X^1, \O_X)$
  is locally free. Then $X$ is smooth.
\end{conj}

Recently, several new special cases of Conjecture~\ref{conj:LZ} have been proved~\cite{GKKP11, Kal11, Dru13, Jor13, GK13, GK14}. One of these (the klt case) was applied in~\cite{GKP13} as well as in~\cite{LT14} to characterize torus quotients among varieties with klt singularities by the vanishing of the first and second Chern class.
It is natural to ask about a similar smoothness criterion for the higher (reflexive) exterior powers of the tangent sheaf, i.e.~for $\sHom_{\O_X}(\Omega_X^p, \O_X)$. In this paper, we therefore propose and study the following problem, which to the best of our knowledge has not appeared in the literature so far.

\begin{ques}[Generalized Lipman--Zariski problem] \label{ques:gen LZ}
Let $X$ be an $n$-dimensional complex variety such that for some integer $1 \le p \le n - 1$, the reflexive hull $\Omega_X^{[p]}$ of the sheaf of K\"ahler $p$-forms is locally free.
Under what assumptions on $X$, and for which values of $p$, does this imply that $X$ is smooth?
\end{ques}

Note that for any variety $X$, the tangent sheaf $\sT_X$ is isomorphic to the dual of $\Omega_X^{[1]}$. So Conjecture~\ref{conj:LZ} is equivalent to asking whether the local freeness of $\Omega_X^{[1]}$ implies that $X$ is smooth. Hence Question~\ref{ques:gen LZ} really is a direct generalization of Conjecture~\ref{conj:LZ}.

As less ambitious intermediate steps, we might also consider the following questions.

\begin{ques}[Weak generalized Lipman--Zariski problem] \label{ques:weak gen LZ}
Let $X$ be an $n$-dimen\-sional complex variety such that for some integer $1 \le p \le n$, the sheaf $\Omt_X^p$ of K\"ahler $p$-forms modulo torsion is locally free.
Under what assumptions on $X$, and for which values of $p$, does this imply that $X$ is smooth?
\end{ques}

\begin{ques}[Very weak generalized Lipman--Zariski problem] \label{ques:very weak gen LZ}
Let $X$ be an $n$-dimensional complex variety such that for some integer $p \ge 1$, the sheaf $\Omega_X^p$ of K\"ahler $p$-forms is locally free.
Under what assumptions on $X$, and for which values of $p$, does this imply that $X$ is smooth?
\end{ques}

\subsection{Main results}

Since even Conjecture~\ref{conj:LZ} is known to fail in positive characteristic~\cite[\S 7]{Lip65}, we work over the field of complex numbers throughout.
Furthermore, we mostly work in the analytic category (i.e.~with complex spaces), since this provides a more natural setting and greater generality. For the definition of the sheaf of K\"ahler differentials in this context, see Section~\ref{sec:Kahler diff}. We also explain there why our results in particular apply to algebraic varieties (Remark~\ref{rem:alg var}).

Our results can be divided in two groups: singularities of the Minimal Model Program and hypersurface singularities. For the definition of the singularities of the MMP, such as klt, terminal, and log canonical, we refer to~\cite[Sec.~2.3]{KM98}.

\subsubsection{Singularities of the MMP}

Somewhat contrary to our expectations, we show that even for low-dimensional terminal singularities, the answer to Question~\ref{ques:gen LZ} is negative for all reasonable values of $p$.

\begin{prp}[Terminal singularities with free sheaves of reflexive differentials] \label{prp:gen LZ MMP}
\hspace{.8em} a) The cone over the second Veronese embedding $\P^2 \subset \P^5$ is a terminal threefold singularity $(x_1 \in X_1)$ such that $\Omega_{X_1}^{[2]}$ is free.

b) The cone over the second Veronese embedding $\P^3 \subset \P^9$ is a four-dimensional isolated terminal Gorenstein singularity $(x_2 \in X_2)$ such that $\Omega_{X_2}^{[2]}$ is free.
\end{prp}

Note that for a (non-smooth) terminal singularity $(x \in X)$, the sheaf $\Omega_X^{[1]}$ cannot be free by~\cite[Cor.~1.3]{GK13}. Likewise, for a four-dimensional terminal Gorenstein singularity $(x \in X)$, the sheaf $\Omega_X^{[3]}$ cannot be free by Lemma~\ref{lem:wedge pairing} below.

We do however have the following result for forms of degree just below the dimension. Loosely speaking, it says that the answer to Question~\ref{ques:gen LZ} for $p = n - 1$ is ``yes up to finitely many `obvious' counterexamples'', if you believe in the original Lipman--Zariski conjecture.

\begin{thm}[Generalized Lipman--Zariski problem for $(n - 1)$-forms] \label{thm:gen LZ n-1}
Let $(x \in X)$ be a normal $n$-dimensional singularity. Assume any of the following.
\begin{enumerate}
\item\label{itm:gen LZ n-1.1} The Lipman--Zariski conjecture holds in dimension $n$.
\item\label{itm:gen LZ n-1.2} The pair $(X, \Delta)$ is log canonical for some $\R$-divisor $\Delta$ on $X$.
\item\label{itm:gen LZ n-1.3} The singular locus of $X$ has codimension at least three.
\end{enumerate}
If $\Omega_X^{[n-1]}$ is free, then $(x \in X)$ is the cone over the $r$-th Veronese embedding of $\P^{n-1}$, for some integer $r$ that divides $n - 1$. Conversely, any such singularity has the property that $\Omega_X^{[n-1]}$ is free.

In particular, in each dimension there are only finitely many log canonical singularities $(x \in X)$ such that $\Omega_X^{[n-1]}$ is free, and these are all isolated and terminal.
\end{thm}

Building on the recent work of~\cite{GKP13} about \'etale fundamental groups, we obtain the following corollary, which for smooth projective $X$ is a classical result from differential geometry.

\begin{cor}[Criterion for quotient singularities and torus quotients] \label{cor:torus quot}
Let $X$ be a normal $n$-dimensional variety such that the pair $(X, \Delta)$ is klt for some $\R$-divisor $\Delta$ on $X$. If the sheaf $\Omega_{X_\sm}^{n-1}$ is flat, then $X$ has at worst quotient singularities.
If $X$ is additionally projective, then there exists an Abelian variety $A$ and a finite surjective Galois morphism $A \to X$ that is \'etale in codimension one.
\end{cor}

Here $X_\sm \subset X$ denotes the smooth locus of $X$. A locally free sheaf $\sF$ on an algebraic variety $X$ is called flat if its analytification $\sF^\an$ is defined by a linear representation of the fundamental group $\pi_1(X^\an)$.

\subsubsection{Hypersurface singularities}

We see that in order to obtain a completely positive answer to Question~\ref{ques:gen LZ}, we must not measure the mildness of the singularity $(x \in X)$ in terms of discrepancies.
Instead, we will focus on a more classically studied case: the class of hypersurface singularities. Note that the examples in Proposition~\ref{prp:gen LZ MMP} are not even complete intersection singularities. (This follows from~\cite{GR11}, for example.)

For normal hypersurfaces $X$, we obtain a positive answer to Question~\ref{ques:gen LZ} for any value of $p$ as soon as the singular locus of $X$ has codimension at least three.
Concerning Question~\ref{ques:weak gen LZ}, we even obtain a positive answer without any extra hypotheses.

\begin{thm}[Generalized Lipman--Zariski problem for hypersurfaces] \label{thm:gen LZ hypersurf I}
Let $(x \in X)$ be a hypersurface singularity of dimension $n$ such that the singular locus of $X$ has codimension at least three.
If $\Omega_X^{[p]}$ is free for some $1 \le p \le n-1$, then $(x \in X)$ is smooth.
\end{thm}

\begin{thm}[Weak generalized Lipman--Zariski problem for hypersurfaces] \label{thm:weak gen LZ hypersurf}
Let $(x \in X)$ be a normal hypersurface singularity of dimension $n$. If $\Omt_X^p$ is free for some $1 \le p \le n$, then $(x \in X)$ is smooth.
\end{thm}

Finally, we can answer Question~\ref{ques:very weak gen LZ} in the affirmative without any assumptions on the singularities. This seemingly innocuous generalization of a well-known smoothness criterion will turn out to be critical in the proofs of Theorems~\ref{thm:gen LZ hypersurf I} and~\ref{thm:weak gen LZ hypersurf}.

\begin{thm}[Very weak generalized Lipman--Zariski problem] \label{thm:very weak gen LZ}
Let $(x \in X)$ be any singularity of embedding dimension $e$. If $\Omega_X^p$ is free for some $1 \le p \le e$, then $(x \in X)$ is smooth.
\end{thm}

\subsection{Outline of proofs}

For Theorem~\ref{thm:gen LZ n-1}, we note that under the given assumptions, $K_X$ is \Q-Cartier. This means that some positive multiple $rK_X$ is a Cartier divisor. Now a standard construction in birational geometry, the so-called \emph{index one cover}~\cite[Def.~5.19]{KM98}, produces a finite surjective morphism $\wt X \to X$ which is the quotient by an action of $\Z/(r)$ and such that $K_{\wt X}$ is Cartier. We prove that in our situation, $\wt X$ is smooth. Hence $(x \in X) \isom \factor{(0 \in \C^n)}{\Z/(r)}$ is a cyclic quotient singularity. We then show that in order for $\Omega_X^{[n-1]}$ to be free, all eigenvalues of the action of $\Z/(r)$ must be equal. It is however well-known that such singularities also have a description as cones over Veronese embeddings.

The technical core of our investigation of hypersurface singularities is the following result, which might be of interest for its own sake.

\begin{thm}[Torsion and cotorsion of a hypersurface] \label{thm:co-tor h}
Let $X = \{ f = 0 \} \subset Y$ be a normal hypersurface, where $f$ is a non-constant holomorphic function on the connected open set $Y \subset \C^{n+1}$. Consider the following complex $(K, \alpha^p)$ of sheaves on $X$:
\[ K\!: 0 \lto \Omega_Y^0|_X \lto \Omega_Y^1|_X \lto
\cdots \lto \Omega_Y^n|_X \lto \Omega_Y^{n+1}|_X \lto 0, \]
where $\alpha^p\!: \Omega_Y^p|_X \to \Omega_Y^{p+1}|_X$ is given by $\alpha^p(\sigma) = \d f \wedge \sigma$. Then:
\begin{enumerate}
\item\label{itm:co-tor h.1} We have isomorphisms $H^p(K) \isom \tor \Omega_X^p$ and $H^p(K) \isom \cotor \Omega_X^{p-1}$, for any $0 \le p \le n + 1$.
\item\label{itm:co-tor h.2} Let $x \in X$ be a singular point, and let $d$ denote the local codimension of the singular locus of $X$ at $x$. Then for the stalks we have
\[ \tor \Omega_{X,x}^p \ne 0 \quad \text{if and only if} \quad d \le p \le n + 1, \]
while
\[ \cotor \Omega_{X,x}^p \ne 0 \quad \text{if and only if} \quad d - 1 \le p \le n. \]
\end{enumerate}
In particular, the following holds:
\begin{enumerate}
\setcounter{enumi}{2}
\item\label{itm:co-tor h.3} If $\cotor \Omega_X^p = 0$ for some $p \le n$, then also $\tor \Omega_X^p = 0$.
\item\label{itm:co-tor h.4} If $\tor \Omega_X^p = 0$ for some $p \le n + 1$, then $\tor \Omega_X^q = 0$ and $\cotor \Omega_X^q = 0$ for all $q \le p - 1$.
\end{enumerate}
\end{thm}

\begin{rem}
We can visualize the information contained in~(\ref{thm:co-tor h}.\ref{itm:co-tor h.2}) in the following table.
\begin{table}[h]
  \centering
    \begin{tabular}{l|cccccccc}
      $p$ & $0$ & \dots & $d-2$ & $d-1$ & $d$ & \dots & $n$ & $n+1$ \\ \hline
      $\tor \Omega_{X,x}^p$   & $0$ & \dots & $0$ & $0$ & $\ne 0$ & \dots & $\ne 0$ & $\ne 0$ \\
      $\cotor \Omega_{X,x}^p$ & $0$ & \dots & $0$ & $\ne 0$ & $\ne 0$ & \dots & $\ne 0$ & $0$ \\
    \end{tabular}
\end{table}
\end{rem}

\begin{rem}
Theorem~\ref{thm:co-tor h}.\ref{itm:co-tor h.2} is a strengthening of a result of Vetter about complete intersections~\cite[Satz 4]{Vet70} in the hypersurface case, and we use Vetter's result in the proof of (\ref{thm:co-tor h}.\ref{itm:co-tor h.2}). However, for our applications we only need~(\ref{thm:co-tor h}.\ref{itm:co-tor h.3}) and~(\ref{thm:co-tor h}.\ref{itm:co-tor h.4}), which could be proved without referring to~\cite{Vet70}.
\end{rem}

\begin{rem}
It follows immediately from (\ref{thm:co-tor h}.\ref{itm:co-tor h.1}) and its proof that for any hypersurface $X = \{ f = 0 \} \subset Y \subset \C^{n+1}$, we have an isomorphism
\[ \tor \Omega_X^p \isom \cotor \Omega_X^{p-1}, \]
given explicitly by
\[ \sigma \mapsto \res_X (f^{-1} \wt\sigma), \]
where $\sigma$ is a local section of $\tor \Omega_X^p$ and $\wt\sigma$ is any lift of $\sigma$ to $\Omega_Y^p|_X$.
For isolated hypersurface singularities, this isomorphism was already observed and stated in a less explicit form by Greuel~\cite[Lemma~2.4, Bemerkung]{Gre80}.
Note, however, that this isomorphism is by no means canonical, since it is not even invariant under multiplication of $f$ by a unit.
\end{rem}

We now briefly explain the proof of Theorem~\ref{thm:gen LZ hypersurf I}. In that setting, Theorem~\ref{thm:co-tor h}.\ref{itm:co-tor h.1} enables us to find a quite short free resolution of $\Omega_X^{p+2}$, thereby bounding the projective dimension of the latter sheaf. Using the well-known relationship between projective dimension, depth, and local cohomology, we are able to conclude that $\Omega_X^{p+2}$ is torsion-free. From (\ref{thm:co-tor h}.\ref{itm:co-tor h.4}), we find that $\Omega_X^p = \Omega_X^{[p]}$. The latter sheaf is free by assumption, and the freeness of a sheaf of K\"ahler differentials is sufficient to conclude that $X$ is smooth by Theorem~\ref{thm:very weak gen LZ}.

Theorem~\ref{thm:weak gen LZ hypersurf} is an immediate consequence of~(\ref{thm:co-tor h}.\ref{itm:co-tor h.3}).

\subsection{Open problems}

It is tempting to ask whether some variant of Theorem~\ref{thm:gen LZ n-1} also holds for intermediate values of $p$, that is, $2 \le p \le n - 2$.
Such a result might allow us to generalize Corollary~\ref{cor:torus quot} to the case where $\Omega_{X_\sm}^{n-1}$ is replaced by $\Omega_{X_\sm}^p$. See Remark~\ref{rem:887} below for the projective case.

Another natural question that poses itself is whether Theorem~\ref{thm:gen LZ hypersurf I} is still true if one drops the assumption that the singular locus has codimension at least three, or maybe even without the normality hypothesis.

Finally, one might ask whether Theorems~\ref{thm:gen LZ hypersurf I} and~\ref{thm:weak gen LZ hypersurf} still hold for complete intersection singularities. We conjecture that this is indeed the case.

\subsection*{Acknowledgements}

I would like to thank Clemens J\"order, S\'andor Kov\'acs and Thomas Peternell for sharing their insight with me.
In particular, S\'andor Kov\'acs on MathOverflow suggested a proof of Theorem~\ref{thm:very weak gen LZ} to me.
Furthermore, I would like to thank the anonymous referee for useful suggestions improving the readability of this paper.

\section{Preparations}

In the present section, we fix our notation and for the reader's convenience, we recall some known results from commutative and homological algebra which we will need.

\subsection{Global conventions}

We work over the field of complex numbers $\C$. A singularity $(x \in X)$ is a germ of a complex space. In singularity theory, the notation $(X, x)$ is standard, but in birational geometry the notation employed here is widely used (cf.~e.g.~\cite{Rei97, KM98}).
For complex spaces, we follow~\cite{GR84}. For the definition of the singularities of the MMP, such as klt, terminal and log canonical, we refer the reader to~\cite[Sec.~2.3]{KM98}.
All rings are assumed to be commutative and contain an identity element. A semigroup is a set with an associative binary operation. A monoid is a semigroup containing a neutral element.

\subsection{Reflexive sheaves} \label{sec:refl}

Let $\sF$ be a coherent sheaf on a normal variety $X$. We denote the torsion subsheaf of $\sF$ by $\tor \sF$. The sheaf $\sF$ is said to be torsion-free if $\tor \sF = 0$. We define $\check\sF := \sF / \tor \sF$.

The dual sheaf of $\sF$ is denoted $\sF^* := \sHom_{\O_X}(\sF, \O_X)$.
The double dual, or reflexive hull, of $\sF$ is denoted $\sF^{**} := (\sF^*)^*$. We always have a canonical map $\lambda\!: \sF \to \sF^{**}$.
The sheaf $\sF$ is said to be reflexive if $\lambda$ is an isomorphism.
In general, $\lambda$ gives rise to a four-term exact sequence
\[ 0 \lto \tor \sF \lto \sF \lto \sF^{**} \lto \cotor \sF \lto 0, \]
where $\cotor \sF$, the cotorsion sheaf of $\sF$, is by definition the cokernel of $\lambda$.
We always have $\cotor \sF = \cotor \check\sF$.

Since taking stalks commutes with $\sHom$ for coherent sheaves~\cite[Annex, \S 4]{GR84}, for any point $x \in X$ we obtain an exact sequence
\[ 0 \lto \tor(\sF_x) \lto \sF_x \lto (\sF_x)^{**} \lto \cotor(\sF_x) \lto 0. \]
In particular, taking the reflexive hull, torsion and cotorsion all commute with taking stalks.

For any positive integer $m$, we define the $m$-th reflexive tensor power $\sF^{[m]} := (\sF^{\tensor m})^{**}$ and the $m$-th exterior power $\bigwedge^{[m]} \sF := (\bigwedge^m \sF)^{**}$. The determinant of $\sF$ is $\det \sF := \bigwedge^{[r]} \sF$, where $r$ is the rank of $\sF$.

For more technical results on reflexive sheaves which we will use implicitly, see~\cite[Sec.~3]{Gra12}.

\subsection{K\"ahler and reflexive differential forms} \label{sec:Kahler diff}

It is well known that the algebraic definition of K\"ahler differentials~\cite[Ch.~II, Sec.~8]{Har77} is badly behaved in the power series case. For example, $\Omega_{\C\{t\}/\C}^1$ is not countably generated as a $\C\{t\}$-module, where $\C\{t\}$ denotes the ring of convergent power series in one variable.

Instead, one needs to work with the module of \emph{universally finite} K\"ahler differentials. If $A$ is a ring and $B$ is an $A$-algebra, the universally finite K\"ahler differential module of $B$ over $A$ is a finitely generated $B$-module $\wh\Omega_{B/A}^1$ together with an $A$-derivation $d\!: B \to \wh\Omega_{B/A}^1$ satisfying the following universal property: for any finitely generated $B$-module $M$, and for any $A$-derivation $\delta\!: B \to M$, there exists a unique $B$-module homomorphism $h\!: \wh\Omega_{B/A}^1 \to M$ such that $\delta = h \circ d$.

If a complex space $X$ is given as the zero set of finitely many holomorphic functions $f_1, \dots, f_r$ on a smooth complex space $Y$, then we define the sheaf of K\"ahler $1$-forms on $X$ by
\[ \Omega_X^1 := \factor{\Omega_Y^1|_X}{\langle \dif f_1, \dots, \dif f_r \rangle}. \]
Here $\Omega_Y^1$ is the cotangent sheaf of $Y$ as a complex manifold. By~\cite[Satz~1.2]{GK64}, this definition is independent of the choice of embedding $X \subset Y$. We can thus define the sheaf $\Omega_X^1$ for any complex space $X$. By~\cite[Kap.~III, \S 4, S\"atze 4 und 5]{GR71}, we have $\Omega_{X,x}^1 = \wh\Omega_{\O_{X,x}/\C}^1$ for any point $x \in X$.

Now assume that $X$ is normal. The sheaf $\Omega_X^1$ and its exterior powers $\Omega_X^p := \bigwedge^p \Omega_X^1$ are not reflexive in general. We denote the reflexive hull of $\Omega_X^p$ by $\Omega_X^{[p]}$. We have
\[ \Omega_X^{[p]} = i_* \Omega_{X_\sm}^p, \]
where $i\!: X_\sm \inj X$ is the inclusion of the smooth locus of $X$.
As explained in Section~\ref{sec:refl}, there are maps (not a complex, of course)
\[ \tor \Omega_X^p \inj \Omega_X^p \surj \Omt_X^p \inj \Omega_X^{[p]}. \]
So we have four different notions of differential form on a normal complex space $X$. They can be described as follows:
A torsion form is a K\"ahler form which vanishes when restricted to the smooth locus of $X$. To give a reflexive differential form, i.e.~a section of $\Omega_X^{[p]}$, is the same as giving a differential form on the smooth locus of $X$. A reflexive differential is a K\"ahler form modulo torsion if and only if it extends to the ambient space for some (equivalently, any) embedding of $X$ in a smooth space.

\begin{rem}[Algebraic varieties] \label{rem:alg var}
Let $X$ be a complex algebraic variety. As explained in~\cite[App.~B]{Har77}, we can associate to $X$ its analytification $X^\an$, which is the complex space locally cut out by the same equations as $X$. We can also analytify coherent algebraic sheaves on $X$. It is well known that $X$ is smooth if and only if $X^\an$ is smooth and that a coherent sheaf $\sF$ on $X$ is locally free if and only if $\sF^\an$ is locally free on $X^\an$.
Furthermore, it is clear (at least locally) from the discussion above that $\Omega_{X^\an}^1 = (\Omega_X^1)^\an$. More generally, $\Omega_{X^\an}^p = (\Omega_X^p)^\an$ and if $X$ is normal, then $\Omega_{X^\an}^{[p]} = (\Omega_X^{[p]})^\an$.
Hence, any positive answer to Question~\ref{ques:gen LZ} formulated for complex spaces immediately yields the corresponding statement about algebraic varieties (by passing to the analytification).
\end{rem}

\subsection{The residue map} \label{sec:res map}

Let $Y$ be a smooth complex space and $X \subset Y$ a reduced normal hypersurface. Let $U \subset Y$ be the open set whose complement is the singular locus of $X$. Similarly as before, for any integer $p$ we define
\[ \Omega_Y^{[p]}(\log X) := i_* \Omega_U^p(\log X|_U). \]
The sheaf $\Omega_Y^{[p]}(\log X)$ is reflexive. On the open set $U$, we have the residue map
\[ \Omega_U^p(\log X|_U) \to \Omega_{X_\sm}^{p-1}. \]
Pushing this forward to $Y$, we obtain a map
\[ \res_X\!: \Omega_Y^{[p]}(\log X) \to \Omega_X^{[p-1]}. \]
See~\cite[Sec.~6]{Gra12} for more details.

\subsection{Quotient singularities} \label{sec:quot sg}

Here we follow~\cite[Ch.~1]{Ste77} and~\cite[Ch.~II, \S 4]{Rei87}.
A linear automorphism $g \in \GL(n, \C)$ is called a \emph{quasi-reflection} if it has $1$ as an eigenvalue of multiplicity exactly $n - 1$. A finite subgroup $G \subset \GL(n, \C)$ is called \emph{small} if it does not contain any quasi-reflections. For any finite subgroup $G \subset \GL(n, \C)$, we denote by $G_{\mr{big}}$ the normal subgroup of $G$ generated by all quasi-reflections. The quotient $\C^n / G_{\mr{big}}$ is an affine space, and $G / G_{\mr{big}}$ acts on this space without quasi-reflections.
Hence, when dealing with quotient singularities up to isomorphism, we may assume that the group action does not contain any quasi-reflections.

The following result of Steenbrink describes the sheaves of reflexive differentials on a quotient singularity.

\begin{prp}[Reflexive differentials on a quotient] \label{prp:refl quot}
Let $G \subset \GL(n, \C)$ be a small subgroup, and let $\pi\!: \C^n \to X = \C^n / G$ be the quotient map. Then, for any $p \ge 0$, we have an isomorphism of sheaves
\[ \Omega_X^{[p]} \isom (\pi_* \Omega_{\C^n}^p)^G. \]
\end{prp}

\begin{proof}
See \cite[Lemma~1.8]{Ste77}.
\end{proof}

Now let $\mu_r \subset \C^*$ be the group of $r$-th roots of unity. Given $a_1, \dots, a_n \in \Z / r$, we define an action of $\mu_r$ on $\C^n$ by letting $\eps \in \mu_r$ act as
\[ (x_1, \dots, x_n) \mapsto (\eps^{a_1} \cdot x_1, \dots, \eps^{a_n} \cdot x_n). \]
The resulting singularity $\left( 0 \in \C^n / \mu_r \right)$ is said to be \emph{of type $\frac 1 r ( a_1, \dots, a_n)$}.
By standard results on linearization at fixed points and diagonalization, we see that every cyclic quotient singularity is of some type.
Note, however, that the type is not uniquely determined by the singularity.

\subsection{Koszul complexes} \label{sec:koszul}

Let $R$ be a ring, $N = R^n$ a free $R$-module, and $x = (x_1, \dots, x_n) \in N$. The complex
\[ K(x) = K(x_1, \dots, x_n)\!: 0 \lto R \lto N \lto \bigwedge\nolimits^2 N \lto \cdots \lto \bigwedge\nolimits^n N \lto 0, \]
where $d^p\!: \bigwedge^p N \to \bigwedge^{p+1} N$ is given by $d^p(\sigma) = x \wedge \sigma$,
is called the \emph{Koszul complex} associated to $x \in N$.
The following result about Koszul complexes in the local case will be crucial for us.

\begin{thm}[Cohomology of Koszul complexes] \label{thm:coh Koszul}
With notation as above, assume additionally that $R$ is a noetherian local ring with maximal ideal $\mf m \subset R$, and $x_1, \dots, x_n \in \mf m$. If
\[ H^k(K(x_1, \dots, x_n)) = 0 \quad \text{for some $k \le n$,} \]
then
\[ H^j(K(x_1, \dots, x_n)) = 0 \quad \text{for all $j \le k$.} \]
\end{thm}

\begin{proof}
By~\cite[Cor.~17.5]{Eis95}, we have $H^n(K(x_1, \dots, x_n)) \isom R/(x_1, \dots, x_n) \ne 0$. Hence we may assume $k \le n - 1$. In this case, the statement follows from~\cite[Thm.~17.6]{Eis95}.
\end{proof}

\subsection{Depth and projective dimension} \label{sec:depth}

Let $R$ be a ring and $M$ an $R$-module. A sequence $x_1, \dots, x_r \in R$ is called \emph{$M$-regular} if $x_1$ is not a zero divisor in $M$, and $x_i$ is not a zero divisor in $M/(x_1, \dots, x_{i-1})M$ for all $2 \le i \le r$. If $\mf a \subset R$ is an ideal, we define the \emph{depth of $M$ with respect to $\mf a$} to be the maximum length of an $M$-regular sequence contained in $\mf a$. We denote it by $\depth_{\mf a} M$. If $R$ is local with maximal ideal $\mf m$, we set $\depth M := \depth_{\mf m} M$.
These definitions also apply to the ring $R$ considered as a module over itself. We say that a noetherian local ring $R$ is \emph{Cohen--Macaulay} if $\depth R = \dim R$.

The general notion of depth can be reduced to the local case by the following result.

\begin{lem}[Depth with respect to an arbitrary ideal] \label{lem:depth}
Assume that $R$ is noetherian and $M$ is finitely generated. For any ideal $\mf a \subset R$, we have
\[ \depth_{\mf a} M = \inf_{\mf p \supset \mf a} \left\{ \depth M_{\mf p} \right\}. \]
\end{lem}

\begin{proof}
See~\cite[Ch.~II, Cor.~1.22]{BS76}.
\end{proof}

For a sheafified notion of depth, let $X$ be a complex space, $Z \subset X$ a closed analytic subset with ideal sheaf $\mf a \subset \O_X$, and $\sF$ a coherent sheaf on $X$. We define
\[ \depth_Z \sF = \inf_{x \in Z} \left\{ \depth_{\mf a_x} \sF_x \right\}. \]

Now let $R$ be a noetherian local ring and $M$ a finitely generated $R$-module. An exact sequence
\[ 0 \lto F_n \lto F_{n-1} \lto \cdots \lto F_1 \lto F_0 \lto M \lto 0 \]
where the $F_i$ are free $R$-modules, is called a \emph{(finite) free resolution} of $M$ of length $n$. The \emph{projective dimension} of $M$, denoted $\pd_R M$, is defined to be the minimum length of a free resolution of $M$, or $+\infty$ if $M$ does not have a finite free resolution. Depth and projective dimension are related by the following famous theorem:

\begin{thm}[Auslander--Buchsbaum formula] \label{thm:AB}
Let $R$ be a noetherian local ring. For any finitely generated $R$-module $M$ of finite projective dimension, we have
\[ \depth M = \depth R - \pd_R M. \]
\end{thm}

\begin{proof}
See~\cite[Thm.~19.9]{Eis95}.
\end{proof}

\subsection{A result of Vetter}

The following result of Vetter will be used in the proof of Theorem~\ref{thm:co-tor h}.

\begin{thm}[Torsion and cotorsion of a complete intersection] \label{thm:Vetter}
Let $(x \in X)$ be a (non-smooth) complete intersection singularity, and let $d$ denote the local codimension of the singular locus of $X$ at $x$. Then we have
\[ \tor \Omega_{X,x}^p
\begin{cases}
= 0,   & 1 \le p \le d - 1, \\
\ne 0, & p = d,
\end{cases}
\]
as well as
\[ \cotor \Omega_{X,x}^p
\begin{cases}
= 0,   & 1 \le p \le d - 2, \\
\ne 0, & p = \max \{ 1, d - 1 \}.
\end{cases}
\]
\end{thm}

\begin{proof}
See~\cite[Satz 4]{Vet70}.
\end{proof}

\section{Terminal singularities and the case of $(n - 1)$-forms}

In this section, we prove Proposition~\ref{prp:gen LZ MMP}, Theorem~\ref{thm:gen LZ n-1}, and Corollary~\ref{cor:torus quot}.

\subsection{Proof of Proposition~\ref{prp:gen LZ MMP}}

By~\cite[Sec.~4.1]{GR11}, the cone over the second Veronese embedding $\P^2 \subset \P^5$ can also be described as the quotient singularity $\frac 1 2 (1, 1, 1)$.
Likewise, the cone over the second Veronese embedding $\P^3 \subset \P^9$ can be described as the quotient singularity $\frac 1 2 (1, 1, 1, 1)$.
We will discuss these singularities in terms of quotients.

By~\cite[(4.11), Theorem]{Rei87}, both of them are terminal.
Concerning part a) of Proposition~\ref{prp:gen LZ MMP}, let $R = \C[x, y, z]$ be the affine coordinate ring of $\C^3$, such that the action of the generator of $G = \Z / 2$ is given by $x \mapsto -x, y \mapsto -y, z \mapsto -z$.
It is easy to see that the $R^G$-module $(\Omega_{\C^3}^2)^G$ of $G$-invariant $2$-forms is freely generated by
\[ \d x \wedge \d y, \; \d x \wedge \d z, \text{ and } \d y \wedge \d z. \]
Using Proposition~\ref{prp:refl quot}, Claim a) follows.

For Claim b), let $x, y, z, w$ be coordinates on $\C^4$. The same argument as before shows that $\Omega_X^{[2]}$ is free. Furthermore, as $(\Omega_{\C^4}^4)^G$ is a free $R^G$-module generated by
\[ \d x \wedge \d y \wedge \d z \wedge \d w, \]
it follows from Proposition~\ref{prp:refl quot} again that $X$ is Gorenstein. \qed

\subsection{Proof of Theorem~\ref{thm:gen LZ n-1}}

We start with two well-known lemmas.

\begin{lem}[Determinant of an exterior power] \label{lem:det wedge}
Let $X$ be a normal complex space and $\sE$ a reflexive sheaf of rank $r$ on $X$. Then we have
\[ \det \left( \bigwedge\nolimits^{[r-1]} \sE \right) \isom ( \det \sE ) ^{[r-1]}. \]
\end{lem}

\begin{proof}
As both sides are reflexive and the locus where $\sE$ is not locally free has codimension $\ge 3$ in $X$, we may assume that $\sE$ is locally free. Choosing a local trivialization of $\sE$ and comparing the transition matrices of both sides, we reduce to the following assertion:
If $f\!: \C^r \to \C^r$ is a linear endomorphism, then
\[ \det \left( \bigwedge\nolimits^{r-1} f \right) = (\det f)^{r-1}. \]
This is clearly true if $f$ is diagonalizable. By a continuity argument, the equality follows for any endomorphism $f$.
\end{proof}

\begin{lem}[Wedge pairing of reflexive differentials] \label{lem:wedge pairing}
Let $X$ be a normal complex space of dimension $n$. Then for any $1 \le p \le n - 1$, the pairing
\[ \Omega_X^{[p]} \x \Omega_X^{[n-p]} \to \Omega_X^{[n]} \]
given by the wedge product is a perfect pairing.
In particular, if $\Omega_X^{[n]}$ and $\Omega_X^{[p]}$ are locally free, then so is $\Omega_X^{[n-p]}$.
\end{lem}

\begin{proof}
The claim is that the induced map
\[ \Omega_X^{[p]} \to \sHom_{\sO_X}( \Omega_X^{[n-p]}, \Omega_X^{[n]} ) \]
is an isomorphism. As both sides are reflexive, we may assume that $X$ is smooth. In this case, the assertion is well-known, and easy to check by a local computation.
\end{proof}

The following proposition is at the heart of the proof of Theorem~\ref{thm:gen LZ n-1}.

\begin{prp}[Cyclic quotients with free reflexive differentials] \label{prp:quot free}
Let $(x \in X)$ be a cyclic quotient singularity of dimension $n$. If $\Omega_X^{[n-1]}$ is free, then $(x \in X)$ is of type $\frac 1 s (1, \dots, 1)$, where $s$ divides $n - 1$.
\end{prp}

\begin{proof}
Let $(x \in X)$ be of type $\frac 1 r (a_1, \dots, a_n)$, for some positive integer $r$ and $a_1, \dots, a_n \in \Z/r$.
For the rest of the proof, we will discuss $(x \in X)$ in terms of the $\mu_r$-action on $\C^n$ thus defined. Furthermore, we will switch to algebraic language. That is, we let $R = \C[x_1, \dots, x_n]$ be the affine coordinate ring of $\C^n$ and we consider $X = \Spec R^{\mu_r}$, where $x \in X$ corresponds to the maximal ideal $(x_1, \dots, x_n) \cap R^{\mu_r}$.

As explained in Section~\ref{sec:quot sg}, we can make the following standard assumptions when dealing with quotient singularities.

\begin{awlog} \label{ass:qf1}
The action of $\mu_r$ on $\C^n$ is faithful. No element of $\mu_r$ acts as a quasi-reflection.
\end{awlog}

This assumption has the following consequence.

\begin{clm} \label{clm:qf2}
For any $1 \le i \le n$, we have $\langle a_1, \dots, \wh{a_i}, \dots, a_n \rangle = \Z/r$.
\end{clm}

Here, as usual, putting a hat on an element means omitting that element, and $\langle S \rangle$ is the subgroup generated by the subset $S$.

\begin{proof}[Proof of Claim~\ref{clm:qf2}]
For simplicity of notation, let $i = 1$. Assume that the subgroup of $\Z/r$ generated by $a_2, \dots, a_n$ is a proper subgroup, say of order $q < r$. If $\zeta \in \mu_r$ is a primitive $r$-th root of unity, then $\zeta^q \ne 1$ and $\zeta^q$ acts either as the identity or as a quasi-reflection, since $q \cdot a_j = 0$ for $j \ge 2$. This, however, contradicts Assumption~\ref{ass:qf1}.
\end{proof}

If $a_1 + \cdots + \wh{a_i} + \cdots + a_n = 0$ for any $i$, then clearly $a_1 = \cdots = a_n$. By Claim~\ref{clm:qf2}, we have $\langle a_1 \rangle = \Z/r$. Hence after pulling back the action by a suitable automorphism of $\mu_r$, we have $a_1 = 1$. Then $n - 1 = 0 \in \Z/r$, that is, $r$ divides $n - 1$, and we are done. So we may assume that $a_1 + \cdots + \wh{a_i} + \cdots + a_n \ne 0$ for some $i$. Re-indexing so that $i = 1$ gives us

\begin{awlog} \label{ass:qf3}
We have $a_2 + \cdots + a_n \ne 0$.
\end{awlog}

We will show that under these assumptions, $\Omega_X^{[n-1]}$ is not locally free at $x$, finishing the proof.
Consider the monoid $M = \N_0^n$. We think of $M$ as the set of monomials on $\C^n$, where $m = (m_1, \dots, m_n) \in M$ corresponds to $x^m := x_1^{m_1} \cdots x_n^{m_n}$.
We have a partial ordering ``$\le$'' on $M$, defined by setting $m \le m'$ if and only if $m_i \le m_i'$ for all indices $i$. Note that $m \le m'$ if and only if $x^{m'}$ is a multiple of $x^m$ in $R$.
 Define $\alpha\!: M \to \Z/r$ by
\[ \alpha(m_1, \dots, m_n) = \sum m_i a_i. \]
An element $\eps \in \mu_r$ acts on the monomial $x^m \in R$ by
\[ x^m \mapsto \eps^{\alpha(m)} x^m. \]
For $1 \le i \le n$, define
\[ \Gamma_i = \big\{ m \in M \;\big|\; \alpha(m) = -(a_1 + \cdots + \wh{a_i} + \cdots + a_n) \big\} \subset M, \]
a non-empty sub-semigroup. By Proposition~\ref{prp:refl quot}, the space of sections $\Hn(X, \Omega_X^{[n-1]})$ has a $\C$-vector space basis
\[ \left\{ x^m \cdot \dif x_1 \wedge \cdots \wedge \wh{\dif x_i} \wedge \cdots \wedge \dif x_n \;\middle|\; 1 \le i \le n, \; m \in \Gamma_i \right\}. \]
Hence as an $R^{\mu_r}$-module, $\Hn(X, \Omega_X^{[n-1]})$ is minimally generated by
\begin{sequation} \label{eqn:mgs}
\left\{ x^m \cdot \dif x_1 \wedge \cdots \wedge \wh{\dif x_i} \wedge \cdots \wedge \dif x_n \;\middle|\; 1 \le i \le n, \; m \in \Gamma_i \text{ minimal} \right\}.
\end{sequation}%
Localizing at $x$, we see that~\eqref{eqn:mgs} also is a minimal generating set for $\Omega_{X,x}^{[n-1]}$ as an $\O_{X,x}$-module.

\begin{clm} \label{clm:qf4}
The semigroup $\Gamma_1$ contains at least two distinct minimal elements.
\end{clm}

\begin{proof}[Proof of Claim~\ref{clm:qf4}]
If $\Gamma_1$ contained just one minimal element, then that element would be the smallest element. We will show that $\Gamma_1$ does not contain a smallest element. We have $0 \not\in \Gamma_1$ by Assumption~\ref{ass:qf3}. Let $m \in \Gamma_1$ be arbitrary. Then $m \ne 0$, say $m_i \ne 0$. By Claim~\ref{clm:qf2}, there exists $m' \in \Gamma_1$ such that $m_i' = 0$. Hence $m \not\le m'$, and $m$ is not the smallest element of $\Gamma_1$.
\end{proof}

It follows from Claim~\ref{clm:qf4} that the set~\eqref{eqn:mgs} contains at least $n + 1$ distinct elements. But the rank of $\Omega_X^{[n-1]}$ at a smooth point of $X$ is only $n$. By Lemma~\ref{lem:stalk fct free} below, the sheaf $\Omega_X^{[n-1]}$ cannot be free at $x$.
\end{proof}

\begin{proof}[Proof of Theorem~\ref{thm:gen LZ n-1}]
Let $(x \in X)$ be a normal $n$-dimensional singularity such that $\Omega_X^{[n-1]}$ is free. By Lemma~\ref{lem:det wedge}, $(n - 1)K_X$ is Cartier. Hence $K_X$ is \Q-Cartier. Let $\pi\!: \wt X \to X$ be the associated index one cover, cf.~\cite[Def.~5.19]{KM98}. The sheaf $\Omega_{\wt X}^{[n]}$ is free. Since $\pi$ is finite and surjective, we have a pull-back map
\[ \pi^*\!: \pi^* \Omega_X^{[n-1]} \to \Omega_{\wt X}^{[n-1]}. \]
Over the smooth locus of $X$, the map $\pi$ is \'etale and hence $\pi^*$ is an isomorphism there. Since both domain and codomain of $\pi^*$ are reflexive, we see that $\pi^*$ is an isomorphism. It follows that $\Omega_{\wt X}^{[n-1]}$ is free. By Lemma~\ref{lem:wedge pairing}, $\Omega_{\wt X}^{[1]}$ is free, hence so is $\sT_{\wt X} = \left( \Omega_{\wt X}^{[1]} \right)^*$.

In each of the cases (\ref{thm:gen LZ n-1}.\ref{itm:gen LZ n-1.1})--(\ref{thm:gen LZ n-1}.\ref{itm:gen LZ n-1.3}), it now follows that $\wt X$ is smooth:
In case (\ref{thm:gen LZ n-1}.\ref{itm:gen LZ n-1.1}), this is immediate from the Lipman--Zariski conjecture in dimension $n$.
In case (\ref{thm:gen LZ n-1}.\ref{itm:gen LZ n-1.2}), we note that the pair $(\wt X, \pi^* \Delta)$ is again log canonical by~\cite[Prop.~5.20]{KM98}. Then we apply~\cite[Cor.~1.3]{GK13}.
And in case (\ref{thm:gen LZ n-1}.\ref{itm:gen LZ n-1.3}), the codimension of the singular locus of $\wt X$ is again at least three because $\pi$ is \'etale over the smooth locus of $X$. We now apply~\cite[Corollary]{Fle88}.

Since $\pi\!: \wt X \to X$ is a quotient map with cyclic Galois group, $(x \in X)$ is now exhibited as a cyclic quotient singularity. By Proposition~\ref{prp:quot free}, $(x \in X)$ is of type $\frac 1 r (1, \dots, 1)$, where $r$ divides $n - 1$. As explained in~\cite[Sec.~4.1]{GR11}, such an $(x \in X)$ can also be described as the cone over the $r$-th Veronese embedding of $\P^{n-1}$.
Conversely, the argument from the proof of Proposition~\ref{prp:gen LZ MMP}.a) shows that for any $n$-dimensional cyclic quotient singularity $(x \in X)$ of type $\frac 1 r (1, \dots, 1)$, where $r$ divides $n - 1$, we have that $\Omega_X^{[n-1]}$ is free.

The last statement of Theorem~\ref{thm:gen LZ n-1} is an immediate consequence of the facts just proved. The singularities in question are terminal by~\cite[(4.11), Theorem]{Rei87}.
\end{proof}

\subsection{Proof of Corollary~\ref{cor:torus quot}}

We rely on the following result of~\cite{GKP13}, which we recall here for the reader's convenience.

\begin{thm}[Extension of flat sheaves] \label{thm:ext flat}
Let $X$ be a normal quasi-projective variety such that the pair $(X, \Delta)$ is klt for some $\R$-divisor $\Delta$ on $X$. Then there exists a normal variety $\wt X$ and a finite surjective Galois morphism $\gamma\!: \wt X \to X$, \'etale in codimension one, such that the following holds:
If $\sG^\circ$ is any flat, locally free, analytic sheaf on $\wt X_\sm^\an$, then there is a flat, locally free, algebraic sheaf $\sG$ on $\wt X$ such that $\sG^\circ$ is isomorphic to $\big( \sG|_{\wt X_\sm} \big)^\an$.
\end{thm}

\begin{proof}
See~\cite[Thm.~1.13]{GKP13}.
\end{proof}

\begin{proof}[Proof of Corollary~\ref{cor:torus quot}]
Let $\gamma_1\!: X_1 \to X$ be a cover enjoying the properties of Theorem~\ref{thm:ext flat}.
By~\cite[Prop.~5.20]{KM98}, the pair $(X_1, \gamma_1^* \Delta)$ is klt. Set $X_1^\circ := \gamma_1^{-1} (X_\sm)$. The flat sheaf $\gamma_1^* \big( \Omega_{X_\sm}^{n-1} \big)$ on $X_1^\circ$ extends to a flat sheaf $\sF$ on all of $X_1$. Since $\gamma_1^* \big( \Omega_{X_\sm}^{n-1} \big) \isom \Omega_{X_1^\circ}^{n-1}$, we get $\sF \isom \Omega_{X_1}^{[n-1]}$ by reflexivity.
In particular, $\Omega_{X_1}^{[n-1]}$ is locally free. Applying Theorem~\ref{thm:gen LZ n-1} shows that $X_1$ has quotient singularities. Hence there exists a finite surjective Galois map $\gamma_2\!: X_2 \to X_1$, \'etale in codimension one, such that $X_2$ is smooth. Let $\gamma_3\!: X_3 \to X_2$ be a morphism such that $\gamma := \gamma_1 \circ \gamma_2 \circ \gamma_3$ is the Galois closure of $\gamma_1 \circ \gamma_2$, cf.~\cite[Thm.~B.1]{GKP13}. We obtain a diagram as follows.
\[ \xymatrix{
X_3 \ar[r]_{\gamma_3} \ar@/^1pc/[rrr]^\gamma & X_2 \ar[r]_{\gamma_2} & X_1 \ar[r]_{\gamma_1} & X
} \]
Since $\gamma_1 \circ \gamma_2$ is \'etale in codimension one, so is $\gamma_3$~\cite[(B.1.2)]{GKP13}. Because $X_2$ is smooth, $\gamma_3$ is \'etale by purity of branch locus. So $X_3$ is smooth, too. Thus the Galois morphism $\gamma\!: X_3 \to X$ exhibits $X$ as having quotient singularities.

For the second claim, we run the same argument as before, and we observe that Theorem~\ref{thm:gen LZ n-1} even gives us that $X_1$ has isolated terminal singularities. In particular, $X_1$ has klt singularities and it is smooth in codimension two. Since $\Omega_{X_1}^{[n-1]}$ is flat, the canonical divisor $K_{X_1}$ is numerically trivial. Furthermore, we have
\[ \sT_{X_1} \isom \Omega_{X_1}^{[n-1]} \tensor \omega_{X_1}^*, \]
so $c_2(\sT_{X_1}) = c_2 \big( \Omega_{X_1}^{[n-1]} \big) = 0$ in the sense of~\cite[Thm.~1.16]{GKP13}. Thus the assumptions of that theorem are satisfied and we may apply it to conclude that there is a finite surjective Galois morphism $A \to X_1$, \'etale in codimension one, where $A$ is an Abelian variety. If $\wt A \to A \to X$ is the Galois closure of the induced map $A \to X$, then by the same argument as above, $\wt A \to A$ is \'etale. Hence $\wt A$ is again an Abelian variety, and the claim is proven.
\end{proof}

\begin{rem} \label{rem:887}
It follows from recent work of Lu and Taji~\cite{LT14} that the second part of Corollary~\ref{cor:torus quot} holds with $\Omega_{X_\sm}^{n-1}$ replaced by $\Omega_{X_\sm}^p$, for any $1 \le p \le n - 1$.
One argues as follows: let $\gamma_1\!: X_1 \to X$ be the cover from Theorem~\ref{thm:ext flat}. Then $\Omega_{X_1}^{[p]}$ is locally free and flat. This implies that $K_{X_1}$ is numerically trivial. By elementary but tedious calculation,
\[ 0 = \wh c_2(\Omega_{X_1}^{[p]}) = - \binom{\binom{n-1}{p-1}}{2} \cdot \underbrace{\wh c_1^2(\sT_{X_1})}_{= 0} + \underbrace{\binom{n-2}{p-1}}_{\ne 0} \cdot \wh c_2(\sT_{X_1}), \]
so $\wh c_2(\sT_{X_1}) = 0$ as well. Here by $\wh c_i$ we denote the \Q-Chern classes, or orbifold Chern classes.
By~\cite[Thm.~1.2]{LT14}, we have that $X_1$ is a torus quotient. Taking Galois closure as above shows that also $X$ is a torus quotient, ending the proof.
\end{rem}

\section{The very weak generalized LZ problem}

It is well known that $\Omega_X^1$ being locally free characterizes smooth spaces.
Theorem~\ref{thm:very weak gen LZ}, which we will prove in the present section, asserts that the same is true of $\Omega_X^p$, with $p$ in a suitable range.

\begin{dfn}[Corank function of a sheaf]
Let $\sF$ be a coherent sheaf on the complex space $X$. We define the \emph{corank function} $\phi_\sF\!: X \to \Z$ of $\sF$ by
\[ \phi_\sF(x) = \dim_{\C} \factor{\sF_x}{\mf m_x \sF_x}, \]
where $\mf m_x \subset \O_{X,x}$ is the maximal ideal.
\end{dfn}

By Nakayama's lemma, $\phi_\sF(x)$ is the cardinality of any minimal generating set of the $\O_{X,x}$-module $\sF_x$.

From now on, assume additionally that $X$ is reduced and connected.

\begin{lem}[Corank function detects local freeness] \label{lem:stalk fct free}
The sheaf $\sF$ is locally free if and only if $\phi_\sF$ is constant.
\end{lem}

\begin{proof}
See~\cite[Ch.~III, Lemma~1.6.i)]{BS76}.
\end{proof}

\begin{lem}[Corank functions of exterior powers] \label{lem:stalk fct wedge}
For any $r \in \N$ and $x \in X$,
\[ \phi_{\bigwedge^r \sF}(x) = \binom{\phi_\sF(x)}{r}. \]
\end{lem}

\begin{proof}
Let $x \in X$ be arbitrary, and let $i\!: \Spec \C \to \Spec \O_{X,x}$ be the natural map.
Consider $\sF_x$ as a coherent sheaf on $\Spec \O_{X,x}$.
By~\cite[Ch.~II, Ex.~5.16.e)]{Har77},
\[ i^*\left( \bigwedge\nolimits^r \sF_x \right) \isom \bigwedge\nolimits^r (i^* \sF_x). \]
The claim follows by looking at the dimensions of these $\C$-vector spaces.
\end{proof}

\begin{prp}[Locally free exterior powers] \label{prp:loc free wedge}
If $\bigwedge^r \sF$ is locally free for some $r \le \sup \phi_\sF(X)$, then $\sF$ itself is locally free.
\end{prp}

\begin{proof}
By looking at some $x \in X$ with $\phi_\sF(x) \ge r$ and using Lemma~\ref{lem:stalk fct wedge}, we see that $\phi_{\bigwedge^r \sF}$ is constant and nonzero.
Hence $\phi_\sF \ge r$ everywhere. However, since
\[ \binom{n+1}{r} = \binom{n}{r} + \binom{n}{r-1} > \binom{n}{r} \]
for $n \ge r$, the map $n \mapsto \binom n r$ is injective on $\{ r, r+1, \dots \}$.
It follows that $\phi_\sF$ is constant.
Now by Lemma~\ref{lem:stalk fct free}, $\sF$ is locally free.
\end{proof}

\begin{proof}[Proof of Theorem~\ref{thm:very weak gen LZ}]
Let $(x \in X)$ be a singularity such that $\Omega_X^p$ is free for some $1 \le p \le e$, where $e = \dim_{\C} \mf m_x / \mf m_x^2$ is the embedding dimension of $(x \in X)$. Then $e = \max \phi_{\Omega_X^1}(X)$
by~\cite[Kap.~III, \S 4.4, Folgerung]{GR71}.
So by Proposition~\ref{prp:loc free wedge}, $\Omega_X^1$ is locally free.
Now~\cite[Kap.~III, \S 4, Satz~7]{GR71} applies to show that $(x \in X)$ is smooth.
\end{proof}

\section{The sheaves of K\"ahler differentials on a hypersurface}

The goal of this section is to prove Theorem~\ref{thm:co-tor h}.
We use the notation from that theorem.
In particular, $X = \{ f = 0 \} \subset Y$ denotes a normal hypersurface in an open set $Y \subset \C^{n+1}$.
If $i\!: X \inj Y$ is the inclusion, we have pull-back maps of K\"ahler differentials
\[ i^*\!: \Omega_Y^p|_X \to \Omega_X^p, \]
for any positive integer $p$.

Whenever we write something like ``let $s \in \Hn(U, \sF)$ be a section'', where $\sF$ is a sheaf on a space $Z$, it is understood that $U$ denotes an arbitrary open subset of $Z$.

\begin{lem}[Logarithmic poles] \label{lem:log poles}
Let $\sigma \in \Hn(U, \Omega_Y^p)$ be a differential form. Then $f^{-1} \cdot \sigma \in \Hn(U, \Omega_Y^{[p]}(\log X))$ if and only if $\d f \wedge \sigma \in \ker \big( \Omega_Y^{p+1} \to \Omega_Y^{p+1}|_X \big)$.
\end{lem}

\begin{proof}
By definition, $f^{-1} \sigma \in \Hn(U, \Omega_Y^{[p]}(\log X))$ if and only if the meromorphic differential forms $f^{-1} \sigma$ and $\d(f^{-1} \sigma)$ both have at most a simple pole along $X$. This is clearly true of $f^{-1} \sigma$, so the condition is that
\[ f \cdot \d(f^{-1} \sigma) = -f^{-1} \d f \wedge \sigma + \d\sigma \]
be a holomorphic differential form. This is the case if and only if $\d f \wedge \sigma$ is a multiple of $f$, i.e.~if the image of $\d f \wedge \sigma$ in $\Omega_Y^{p+1}|_X$ is zero.
\end{proof}

\begin{lem}[Factorization] \label{lem:factorization}
Let $R$ be a ring, $I \subset R$ an ideal, and $\phi\!: M \to N$ a map of $R$-modules. Furthermore, let $P$ be an $R/I$-module and $\wt\rho\!: \phi^{-1}(IN) \to P$ an $R$-linear map.
If $IM \subset \ker(\wt\rho)$, then $\wt\rho$ factorizes via a map
\[ \rho\!: \ker(\overline\phi) \to P, \]
where $\overline\phi = \phi \tensor \id\!: M \tensor_R R/I \to N \tensor_R R/I$.
\end{lem}

\begin{proof}
Note that $M \tensor_R R/I \isom M/IM$, and likewise for $N$. Hence we have a short exact sequence
\begin{sequation} \label{eqn:factorization}
0 \to IM \to \phi^{-1}(IN) \to \ker(\overline\phi) \to 0.
\end{sequation}%
Since we assumed that $IM \subset \ker(\wt\rho)$, $\wt\rho$ induces the desired map $\rho$ by the universal property of quotient modules and the sequence~\eqref{eqn:factorization}.
\end{proof}

\begin{prp}[Residue map] \label{prp:rho}
For any $1 \le p \le n$, we have a map of sheaves on $X$,
\[ \rho_{p+1}\!: \ker \alpha^{p+1} \to \Omega_X^{[p]}, \]
given by $\sigma \mapsto \res_X(f^{-1} \sigma)$.
\end{prp}

\begin{proof}
We would like to apply Lemma~\ref{lem:factorization}. Let $x \in X$ be any point, and set
\begin{align*}
R    &= \O_{Y,x}, \\ 
I    &= (f) \subset R, \\
\phi &= (\sigma \mapsto \d f \wedge \sigma)\!: M = \Omega_{Y,x}^{p+1} \to N = \Omega_{Y,x}^{p+2}, \\
P    &= \Omega_{X,x}^{[p]}, \\
\wt\rho &= \big( \sigma \mapsto \res_X(f^{-1} \sigma) \big)\!: \phi^{-1}(IN) \to P,
\end{align*}
where $\res_X\!: \Omega_{Y,x}^{[p+1]}(\log X) \to \Omega_{X,x}^{[p]}$ is the residue map, cf.~Section~\ref{sec:res map}.
It follows from Lemma~\ref{lem:log poles} that $\wt\rho$ exists. Furthermore, it is straightforward to see that $IM \subset \ker(\wt\rho)$.
Hence we obtain the desired map $\rho_{p+1}$ stalkwise by applying Lemma~\ref{lem:factorization}.
It is clear from the construction that the maps of stalks glue to a map of sheaves.
\end{proof}

\begin{prp}[$K$ computes the cotorsion] \label{prp:cotor}
The map $\rho_{p+1}$ from Proposition~\ref{prp:rho} is an isomorphism, and it induces an isomorphism of short exact sequences
\[ \xymatrix{
0 \ar[r] & \img \alpha^p \ar[d]^\wr \ar[r] & \ker \alpha^{p+1} \ar[d]^\wr_{\rho_{p+1}} \ar[r] & H^{p+1}(K) \ar[d]^\wr \ar[r] & 0 \\
0 \ar[r] & \Omt_X^p \ar[r] & \Omega_X^{[p]} \ar[r] & \cotor \Omega_X^p \ar[r] & 0.
} \]
\end{prp}

\begin{proof}
The inclusion $\ker \alpha^{p+1} \subset \Omega_Y^{p+1}|_X$ exhibits $\ker \alpha^{p+1}$ as a saturated subsheaf of a reflexive (even free) sheaf, because the quotient injects into the torsion-free (even free) sheaf $\Omega_Y^{p+2}|_X$. Hence $\ker \alpha^{p+1}$ is itself reflexive. Furthermore, $\Omega_X^{[p]}$ is reflexive by definition.
So to prove the first claim, it is sufficient to show that $\rho_{p+1}$ is an isomorphism on the smooth locus of $X$. This is easily accomplished by a local computation.

Locally at a smooth point $x \in X \subset Y$, we may choose coordinates $x_0, \dots, x_n$ on $Y$ such that the defining equation of $X \subset Y$ is $f = x_0$. In these coordinates,
\[ \Omega_X^{[p]} = \bigoplus_{\substack{I \subset \{ 1, \dots, n \} \\[.3ex] |I| = p}} \O_X \cdot \d x_I, \]
where we use the usual multi-index notation. Likewise, for any $q$,
\[ \Omega_Y^q|_X = \bigoplus_{\substack{I \subset \{ 0, \dots, n \} \\[.3ex] |I| = q}} \O_X \cdot \d x_I, \]
and the map $\alpha^{p+1}$ is given by
\[ \alpha^{p+1}(\d x_I) =
\begin{cases}
\d x_{I \cup \{ 0 \}}, & 0 \not\in I, \\
0, & 0 \in I.
\end{cases}
\]
So
\[ \ker \alpha^{p+1} = \bigoplus_{\substack{I \subset \{ 0, \dots, n \} \\[.3ex] |I| = p + 1 \\[.3ex] 0 \in I}} \O_X \cdot \d x_I
= \bigoplus_{\substack{I \subset \{ 1, \dots, n \} \\[.3ex] |I| = p}} \O_X \cdot (\d x_0 \wedge \d x_I), \]
and
\[ \rho_{p+1}(\d x_0 \wedge \d x_I) = \res_X \left( \frac{\d x_0}{x_0} \wedge d x_I \right) = \d x_I. \]
We see that $\rho_{p+1}$ sends a basis to a basis, and hence it is an isomorphism locally around $x$.

Now we turn to the second claim. We will show that the left vertical map exists and is an isomorphism. The existence and isomorphy of the right vertical map then immediately follows from the Five lemma.

So let $\sigma \in \Hn(U, \Omega_Y^p|_X)$ be any form. Then
\begin{sequation} \label{eqn:cotor}
\rho_{p+1}(\alpha^p(\sigma)) = \res_X \left( \frac{\d f}{f} \wedge \sigma \right) = i^* \sigma
\end{sequation}%
is a reflexive form on $X$ which extends to $Y$, i.e.~it is a K\"ahler form modulo torsion. This means that the restriction of $\rho_{p+1}$ to $\img \alpha^p$ factors via the inclusion $\Omt_X^p \subset \Omega_X^{[p]}$. The map $\img \alpha^p \to \Omt_X^p$ thus obtained is the left vertical arrow.

Since this map $\img \alpha^p \to \Omt_X^p$ is automatically injective, all that remains to show is its surjectivity.
Let $\sigma \in \Hn(U, \Omt_X^p)$ be arbitrary. Then $\sigma = i^* \wt\sigma$ for some $\wt\sigma \in \Hn(U, \Omega_Y^p|_X)$. Since
\[ \rho_{p+1}(\alpha^p(\wt\sigma)) = i^* \wt\sigma = \sigma \]
as in~\eqref{eqn:cotor}, $\sigma$ is contained in the image of $\img \alpha^p$.
\end{proof}

\begin{prp}[$K$ computes the torsion] \label{prp:tor}
We have an isomorphism of short exact sequences
\[ \xymatrix{
0 \ar[r] & H^p(K) \ar[d]^\wr \ar[r] & \coker \alpha^{p-1} \ar[d]_{i^*}^\wr \ar[r] & \img \alpha^p \ar[d]^\wr \ar[r] & 0 \\
0 \ar[r] & \tor \Omega_X^p \ar[r] & \Omega_X^p \ar[r] & \Omt_X^p \ar[r] & 0.
} \]
\end{prp}

\begin{proof}
First, note that the map $\coker \alpha^{p-1} \to \img \alpha^p$ in the upper sequence is induced by $\alpha^p$. By definition, we have
\[ \Omega_X^1 = \factor{ \Omega_Y^1|_X }{ \O_X \cdot \d f } \]
and hence~\cite[Prop.~A2.2.d)]{Eis95}
\[ \Omega_X^p = \factor{ \Omega_Y^p|_X }{ \Omega_Y^{p-1}|_X \wedge \d f }. \]
This shows the existence and isomorphy of the middle vertical map.
The right vertical map is the same as the left vertical map in Proposition~\ref{prp:cotor}.
We only need to show that the square formed by these maps (which is the right-hand square in the above diagram) commutes, for then the existence and isomorphy of the left vertical map follows from the Five lemma again.

Any local section of $\coker \alpha^{p-1}$ is represented by some form $\sigma \in \Hn(U, \Omega_Y^p|_X)$. If we go around the upper right corner of the square, we arrive at
\[ \rho_{p+1}(\alpha^p(\sigma)) = i^* \sigma \in \Hn(U, \Omt_X^p) \]
as in~\eqref{eqn:cotor}. If we go around the lower left corner, we clearly get the same.
\end{proof}

\begin{rem}
The horizontal sequences of Propositions~\ref{prp:cotor} and~\ref{prp:tor} are dual to each other in a categorical sense. However, it is not clear to us whether the vertical maps, which are induced by $\rho_{p+1}$ and $i^*$, respectively, are also dual in some sense.
\end{rem}

\begin{proof}[Proof of Theorem~\ref{thm:co-tor h}]
Theorem~\ref{thm:co-tor h}.\ref{itm:co-tor h.1} is immediate from Propositions~\ref{prp:cotor} and~\ref{prp:tor}.
For~(\ref{thm:co-tor h}.\ref{itm:co-tor h.2}), let $x \in X$ be a singular point, and let $K_x$ denote the complex obtained from $K$ by taking stalks at $x$.
Since the stalk functor is exact, from (\ref{thm:co-tor h}.\ref{itm:co-tor h.1}) we get
\[ H^p(K_x) \isom \tor \Omega_{X,x}^p \quad \text{and} \quad H^p(K_x) \isom \cotor \Omega_{X,x}^{p-1}. \]
Now note that $K_x$ is a Koszul complex. More precisely, with respect to the usual basis $\dif x_1, \dots, \dif x_{n+1}$ of the free $\O_{X,x}$-module $\Omega_Y^1|_{X,x}$ we have
\[ K_x = K(\del f/\del x_1, \dots, \del f/\del x_{n+1}), \]
using the notation from Section~\ref{sec:koszul}. Furthermore, since $x \in X$ was assumed to be a singular point, we have $\del f/\del x_i \in \mf m_x \subset \O_{X,x}$ for all $i$. Hence Theorem~\ref{thm:co-tor h}.\ref{itm:co-tor h.2} follows from Theorems~\ref{thm:coh Koszul} and~\ref{thm:Vetter}.

Theorems~\ref{thm:co-tor h}.\ref{itm:co-tor h.3} and~\ref{thm:co-tor h}.\ref{itm:co-tor h.4} follow immediately from~(\ref{thm:co-tor h}.\ref{itm:co-tor h.2}) and the fact that a sheaf is zero if and only if all its stalks are zero.
\end{proof}

\section{The generalized LZ problem for hypersurface singularities}

In this section, we prove Theorem~\ref{thm:gen LZ hypersurf I}.
So let $(x \in X)$ be a hypersurface singularity of dimension $n$ such that the singular locus of $X$ has codimension at least three, and such that $\Omega_X^{[p]}$ is free for some $1 \le p \le n-1$. We aim to show that $(x \in X)$ is smooth. Note that $(x \in X)$ is at least normal by~\cite[Ch.~II, Thm.~8.22A]{Har77}.

\begin{clm} \label{clm:depth est}
Let $Z \subset X$ be a closed analytic subset of codimension $\ge 3$. Then we have
\[ \depth_Z \Omega_X^{p+2} \ge 1. \]
\end{clm}

\begin{proof}[Proof of Claim~\ref{clm:depth est}]
We may assume that $(x \in X)$ is embedded as a hypersurface in an open subset $Y \subset \C^{n+1}$. Thus we are in the situation of Theorem~\ref{thm:co-tor h}, and we will use the notation from that theorem.
By Propositions~\ref{prp:cotor} and~\ref{prp:tor}, the map $\alpha^{p+1}$ gives rise to an exact sequence
\begin{sequation} \label{eqn:resol Omt}
0 \lto \Omega_X^{[p]} \xrightarrow{\dif f \wedge} \Omega_Y^{p+1}|_X \lto \Omega_Y^{p+2}|_X \lto \Omega_X^{p+2} \lto 0,
\end{sequation}%
which in this case we can view as a free resolution of $\Omega_X^{p+2}$ of length $2$.
Now let $x \in Z$ be any point. Denoting the ideal sheaf of $Z$ by $\mf a \subset \O_X$, we need to show
\[ \depth_{\mf a_x} \Omega_{X,x}^{p+2} \ge 1. \]
By Lemma~\ref{lem:depth}, it is enough to show
\[ \depth \big( \Omega_{X,x}^{p+2} \big)_{\mf p} \ge 1 \]
for any prime ideal $\mf a_x \subset \mf p \subset \O_{X,x}$.
By abuse of notation, we write $\O_{X, \mf p}$ for the localization of $\O_{X,x}$ at the prime $\mf p$.
Since $X$ is a hypersurface, the local ring $\O_{X, x}$ is Cohen--Macaulay, and then so is $\O_{X, \mf p}$~\cite[Ch.~II, Thm.~8.21A]{Har77}. This gives us
\[ \depth \O_{X, \mf p} = \dim \O_{X, \mf p} \ge \codim_X Z \ge 3. \]
On the other hand, from~\eqref{eqn:resol Omt} we see that
\[ \pd_{\O_{X, \mf p}} \big( \Omega_{X, x}^{p+2} \big)_{\mf p} \le 2. \]
Hence by the Auslander--Buchsbaum formula, Theorem~\ref{thm:AB}, we obtain
\[ \depth \big( \Omega_{X, x}^{p+2} \big)_{\mf p} = \depth \O_{X, \mf p} - \pd_{\O_{X, \mf p}} \big( \Omega_{X, x}^{p+2} \big)_{\mf p} \ge 1 \]
as desired.
\end{proof}

\begin{clm} \label{clm:tor-free}
We have $\tor \Omega_X^{p+2} = 0$.
\end{clm}

\begin{proof}[Proof of Claim~\ref{clm:tor-free}]
Since a K\"ahler form is torsion if and only if it vanishes on the smooth locus, all we need to prove is the injectivity of the restriction map
\[ H^0(U, \Omega_X^{p+2}) \to H^0(U_\sm, \Omega_X^{p+2}) \]
for any open set $U \subset X$.
This follows immediately from Claim~\ref{clm:depth est} applied to $Z = X_\sg$ and~\cite[Ch.~II, Thm.~3.6]{BS76}.
\end{proof}

\begin{rem}
It is also possible to give an elementary proof of Claim~\ref{clm:tor-free}, directly from the definition of depth.
\end{rem}

We can now conclude the proof of Theorem~\ref{thm:gen LZ hypersurf I}.
By Claim~\ref{clm:tor-free} and Theorem~\ref{thm:co-tor h}.\ref{itm:co-tor h.4},
\[ \tor \Omega_X^p = \cotor \Omega_X^p = 0. \]
This means that $\Omega_X^p = \Omega_X^{[p]}$. The latter sheaf is free by assumption, so Theorem~\ref{thm:very weak gen LZ} applies to show that $(x \in X)$ is smooth.

\section{The weak generalized LZ problem for hypersurface singularities}

This last and very short section is devoted to the proof of Theorem~\ref{thm:weak gen LZ hypersurf}. Given Theorems~\ref{thm:very weak gen LZ} and~\ref{thm:co-tor h}, this proof is very easy. Let $(x \in X)$ be a normal hypersurface singularity of dimension $n$, and assume that $\Omt_X^p$ is free for some $1 \le p \le n$. Then obviously
\[ \cotor \Omega_X^p = \cotor \Omt_X^p = 0. \]
Hence by Theorem~\ref{thm:co-tor h}.\ref{itm:co-tor h.3}, we also get $\tor \Omega_X^p = 0$, so $\Omega_X^p = \Omega_X^{[p]}$. But $\Omega_X^{[p]} = \Omt_X^p$ is free by assumption, so $(x \in X)$ is smooth by Theorem~\ref{thm:very weak gen LZ}. \qed



\providecommand{\bysame}{\leavevmode\hbox to3em{\hrulefill}\thinspace}
\providecommand{\MR}{\relax\ifhmode\unskip\space\fi MR}
\providecommand{\MRhref}[2]{%
  \href{http://www.ams.org/mathscinet-getitem?mr=#1}{#2}
}
\providecommand{\href}[2]{#2}

\end{document}